\documentclass{article}
\usepackage{graphicx,latexsym,amsmath,amsthm,amssymb,color,verbatim,mathrsfs,bbm,amscd}
\usepackage[latin1] {inputenc}
\usepackage[english]{babel}
\usepackage{setspace}
\usepackage[vcentermath]{youngtab}
\usepackage{leftidx}
\usepackage[all]{xy}
\usepackage{tikz}
\usepackage{url}
\usepackage{paralist}
\usetikzlibrary{matrix,decorations.pathreplacing}
\usepackage{floatrow}
\newcommand{\Sym}{\mathsf{Sym}}

\newcommand{\HWV}{\mathsf{HWV}}

\newcommand{\la}{\lambda}
\newcommand{\IC}{\ensuremath{\mathbb{C}}}

\newcommand{\IN}{\ensuremath{\mathbb{N}}}

\newcommand{\aS}{\ensuremath{\mathfrak{S}}}

\DeclareMathOperator{\diag}{diag}

\newcommand{\Specht}[1]{[{#1}]}
\newcommand{\tensor}{\smash{\textstyle\bigotimes}}
\newcommand{\GL}{\mathsf{GL}}

\DeclareMathOperator{\sdet}{eval}

\newcommand{\transpose}[1]{{{^t\!}{#1}}}

\swapnumbers
\numberwithin{equation}{section}
\newtheorem{theorem}[equation]{Theorem}
\newtheorem{corollary}[equation]{Corollary}
\newtheorem{lemma}[equation]{Lemma}

\newcommand{\fref}[1]{\ref{#1}}

\newcommand{\partinto}[1][]{\smash{\mathord{\mathchoice{%
  \xymatrix@=0.4em@1{%
  \ar@{|-}[rr]_-*--{\scriptstyle #1}
  &*{\phantom{\scriptstyle{#1}}}&}
}{
  \xymatrix@=0.25em@1{%
  \ar@{|-}[rr]_-*--{\scriptstyle #1}
  &*{\phantom{\scriptstyle{#1}}}&}
}{
  \xymatrix@=0.2em@1{%
  \ar@{|-}[rr]_-*--{\scriptscriptstyle #1}
  &*{\phantom{\scriptscriptstyle{#1}}}&}
}{}}}}
\newcommand{\partintonosmash}[1][]{\mathord{\mathchoice{%
  \xymatrix@=0.4em@1{%
  \ar@{|-}[rr]_-*--{\scriptstyle #1}
  &*{\phantom{\scriptstyle{#1}}}&}
}{
  \xymatrix@=0.25em@1{%
  \ar@{|-}[rr]_-*--{\scriptstyle #1}
  &*{\phantom{\scriptstyle{#1}}}&}
}{
  \xymatrix@=0.2em@1{%
  \ar@{|-}[rr]_-*--{\scriptscriptstyle #1}
  &*{\phantom{\scriptscriptstyle{#1}}}&}
}{}}}
\newcommand{\partintostar}[1][]{\smash{\mathord{\mathchoice{%
  \xymatrix@=0.4em@1{%
  \ar@{|-}[rr]_-*--{\scriptstyle #1}^-*--{\scriptstyle \ast}
  &*{\phantom{\scriptstyle{#1}}}&}
}{
  \xymatrix@=0.25em@1{%
  \ar@{|-}[rr]_-*--{\scriptstyle #1}^-*--{\scriptstyle \ast}
  &*{\phantom{\scriptstyle{#1}}}&}
}{
  \xymatrix@=0.2em@1{%
  \ar@{|-}[rr]_-*--{\scriptscriptstyle #1}^-*--{\scriptstyle \ast}
  &*{\phantom{\scriptscriptstyle{#1}}}&}
}{}}}}
\newcommand{\partintostarnosmash}[1][]{\mathord{\mathchoice{%
  \xymatrix@=0.4em@1{%
  \ar@{|-}[rr]_-*--{\scriptstyle #1}^-*--{\scriptstyle \ast}
  &*{\phantom{\scriptstyle{#1}}}&}
}{
  \xymatrix@=0.25em@1{%
  \ar@{|-}[rr]_-*--{\scriptstyle #1}^-*--{\scriptstyle \ast}
  &*{\phantom{\scriptstyle{#1}}}&}
}{
  \xymatrix@=0.2em@1{%
  \ar@{|-}[rr]_-*--{\scriptscriptstyle #1}^-*--{\scriptstyle \ast}
  &*{\phantom{\scriptscriptstyle{#1}}}&}
}{}}}

\newcommand{\kron}[3]{k_{{#1}{#2}{#3}}}
\newcommand{\kronk}[3]{k_{{#1},{#2},{#3}}}

\theoremstyle{definition}
\newtheorem{definition}[equation]{Definition}

\newcommand\Zehn{10}

\title{The Saxl Conjecture and the Dominance Order}
\author{Christian Ikenmeyer\thanks{Texas A\&M University, ciken$@$math.tamu.edu}}

\date{2015-May-05}

\begin{document}
\sloppy

\maketitle
\begin{abstract}
In 2012 Jan Saxl conjectured that all irreducible representations of the symmetric group occur in the decomposition of
the tensor square of the irreducible representation corresponding to the staircase partition.
We make progress on this conjecture by proving the occurrence of all those irreducibles which correspond to partitions
that are comparable to the staircase partition in the dominance order.
Moreover, we use our result to show the occurrence of all irreducibles corresponding to hook partitions.
This generalizes results by Pak, Panova, and Vallejo from 2013.
\end{abstract}

\bigskip

{\small

\medskip

}

\section{Introduction}
In their recent work \cite{ppv:13} (see also \cite{ppv:14}) Pak, Panova, and Vallejo study the tensor square conjecture for symmetric groups and the related Saxl conjecture.
The tensor square conjecture states that for all natural numbers $d$ besides 2, 4 and 9
there exists an irreducible representation $\Specht\lambda$ of the symmetric group $\aS_d$ on
$d$ letters such that every irreducible representation of $\aS_d$ is a constituent of the tensor square of $\Specht\lambda$.
Jan Saxl conjectured in 2012 that in the case of $d$ being a triangle number the isomorphy type $\lambda$ can be chosen to be the staircase partition~$\varrho$.
We make progress on the Saxl conjecture by showing that all those partitions
which are comparable to the staircase partition in the dominance order
actually appear in the decomposition of the tensor square of $\Specht\varrho$, see Theorem~\ref{thm:main} below.
As a corollary, we also proof that all hook partitions appear in the decomposition of the tensor square of $\Specht\varrho$, see Corollary~\ref{cor:saxlhooks},
which can also be derived from \cite[Thm.~4.12]{ppv:13} using the same ideas that we use in Section~\fref{sec:hooks}.
Besides hooks, the recent paper \cite{ppv:13} contains partial results about two-row partitions, certain three-row-partitions, and the case of a two-row partition with an additional column.
Our work generalizes the first two of these three cases.
Other work on the Saxl conjecture can be found in \cite{Val:14}.

The aforementioned conjectures are questions about the positivity of certain Kronecker coefficients.
Recently the study of these coefficients has intensified, as they arise prominently in geometric complexity theory (see e.g.~\cite{gct1,gct2,BI:10,BI:13,BLMW:11,ike:12b,pp:14} to name a few) and in quantum information theory (see e.g.~\cite{chm:05,Chr:06,cm:06,Kly:06,chm:07,CDW:12} and references therein).

Our proof of Theorem~\ref{thm:main} uses the interpretation of the Kronecker coefficient as the dimension of the space of homogeneous highest weight polynomials on the triple tensor product space.
Using polarization, a standard method from multilinear algebra, we show that these polynomials do not vanish.

\subsection*{Acknowledgments}
I am grateful to the Simons Institute for the Theory of Computing in Berkeley for hosting me during this research.
I want to thank Cameron Farnsworth, Igor Pak, Greta Panova, and Ernesto Vallejo for helpful discussions.
I also want to thank an anonymous referee for valuable comments.

\section{Preliminaries}
A \emph{partition} $\la$ is defined to be a finite sequence of nonincreasing nonnegative integers $\la = (\la_1, \la_2, \ldots, \la_n)$.
A pictorial description of partitions are \emph{Young diagrams},
which are upper-left-justified arrays having $\la_i$ boxes in the $i$th row, for example the partition $(5,3,1,1)$ can be depicted as follows:
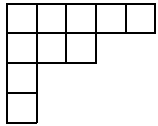
\[\young(\ \ \ \ \ ,\ \ \ ,\ ,\ )\]
The \emph{transposed Young diagram} of $\la$ is obtained by flipping the Young diagram of $\la$ at the main diagonal.
The corresponding partition is denoted $\transpose\la$. For example $\transpose(5,3,1,1)=(4,2,2,1,1)$.
The length of the $i$th row of $\la$ is given by $\la_i$ and
the length of the $i$th column of $\la$ is given by $\transpose\la_i$.
We call $|\la|:=\sum_{i=1}^n \la_i$ the \emph{number of boxes of $\la$}.
If the number of boxes of $\la$ is $d$, then we say that $\la$ is a partition of~$d$.
A partition $\la$ \emph{dominates} another partition $\varrho$ if for all $k$ we have $\sum_{i=1}^k \la_i \geq \sum_{i=1}^k \varrho_i$.
If $\la$ dominates $\varrho$ or $\varrho$ dominates $\la$, we say that \emph{$\la$ and $\varrho$ are comparable in the dominance order}.
Let $\varrho(n):=(n,n-1,n-2,\ldots,1)$ denote the so-called \emph{staircase partition} with $d := {n(n+1)}/2$ boxes.

Our base field are the complex numbers.
Every partition $\la$ with $d$ boxes corresponds to an isomorphy type $\Specht\la$ of irreducible representations of the symmetric group $\aS_d$.
For two partitions $\la$ and $\mu$ of $d$ we have that
the group $\aS_d$ also acts naturally on the tensor product $\Specht\la \otimes \Specht\mu$ by diagonally embedding $\aS_d \hookrightarrow \aS_d \times \aS_d$, $\pi \mapsto (\pi,\pi)$
and this tensor product decomposes into irreducibles.
Our main result is the following theorem.
\begin{theorem}\label{thm:main}
For every partition $\nu$ with ${n(n+1)}/2$ boxes that is comparable in the dominance order to the staircase partition $\varrho(n)$ we have that
the tensor square representation $\Specht{\varrho(n)} \otimes \Specht{\varrho(n)}$ contains the irreducible representation $\Specht\nu$ as an irreducible constituent.
\end{theorem}

Let $\varrho:=\varrho(n)$.
The number of occurrences of $\Specht\nu$ in $\Specht{\varrho} \otimes \Specht{\varrho}$ is called the \emph{Kronecker coefficient} $\kron {\varrho}{\varrho}{\nu}$.
Hence Theorem~\ref{thm:main} states that $\kron {\varrho}{\varrho}{\nu}$ is nonzero for all partitions $\nu$ of $n(n+1)/2$ that are comparable to $\varrho$ in the dominance order.
Our proof uses a different but also well known interpretation of Kronecker coefficients which is also used in \cite{ike:12b,BI:13,HIL:13}.
We explain this description in section~\ref{sec:newinterpretationofkron}.

For proving Theorem~\ref{thm:main} without loss of generality we can assume that $\nu$ dominates $\varrho$:
Indeed, if $\varrho$ dominates $\nu$, then $\transpose\nu$ dominates $\transpose\varrho = \varrho$
and we have $\kron {\varrho}{\varrho}{\nu} = \kron {\varrho}{(\transpose\varrho)}{(\transpose\nu)} > 0$,
because it is well known that the Kronecker coefficient is invariant under transposition of any two of its parameters.
This follows almost immediately from character theory of the symmetric group (and can be found in \cite[Lemma 4.4.7]{ike:12b}).

\section{Highest Weight Vectors}\label{sec:newinterpretationofkron}
Let $\la,\mu,\nu$ be partitions of $d$ with at most $n$ rows.
The group $\GL_n^3 := \GL_n \times \GL_n \times \GL_n$ acts on the third tensor power $\tensor^3\IC^n$ via
\[
(g', g'', g''') (v' \otimes v'' \otimes v''') := (g' v') \otimes (g'' v'') \otimes (g''' v''').
\]
Since $\GL_n^3$ acts on $\tensor^3 \IC^n$, the symmetric power $\Sym^d(\tensor^3\IC^n)$ is a finite dimensional $\GL_n^3$-representation.
Indeed, $\Sym^d(\tensor^3\IC^n) \subseteq \tensor^d(\tensor^3\IC^n)$ is the subrepresentation of tensors that are invariant under permuting the $d$ tensor factors.
For $\alpha \in \IC^n$ let $\diag(\alpha_1,\ldots,\alpha_n)$ denote the diagonal matrix with $\alpha_i$ on the main diagonal, i.e., an element of the maximal torus of $\GL_n^3$ (for the standard basis).
Given a partition triple $(\la,\mu,\nu)$ with $d$ boxes each,
a vector $f \in \Sym^d(\tensor^3\IC^n)$ is called a \emph{weight vector of type $(\la,\mu,\nu)$} if for all triples $(g',g'',g''')$ of diagonal matrices
$g'=\diag(g'_1,g'_2,\ldots,g'_n)$,
$g''=\diag(g''_1,g''_2,\ldots,g''_n)$,
$g'''=\diag(g'''_1,g'''_2,\ldots,g'''_n)$
we have
\[
(g',g'',g''') f = \prod_{i=1}^n (g'_i)^{\la_i} \prod_{i=1}^n (g''_i)^{\mu_i} \prod_{i=1}^n (g'''_i)^{\nu_i} f.
\]
Let $U_n \subseteq \GL_n^3$ denote the subgroup of triples of upper triangular matrices with 1s on the main diagonals, i.e., the maximal unipotent group of $\GL_n^3$.
A weight vector is called a \emph{highest weight vector} if $\forall g \in U_n : g f = f$.
The set of highest weight vectors of a given type $(\la,\mu,\nu)$ in $\Sym^d(\tensor^3\IC^n)$ forms a vector space which we denote by
$\HWV_{\la,\mu,\nu}(\Sym^d(\tensor^3\IC^n))$.
Using Schur-Weyl duality, a small calculation can show that the Kronecker coefficient
$\kron \nu \mu \nu$ is the dimension of $\HWV_{\la,\mu,\nu}(\Sym^d(\tensor^3\IC^n))$, see basically \cite[eq.~(14)]{chm:05}.
This is also worked out in \cite[Sec.~4.4]{ike:12b}.

The main idea is to study the $d$th tensor power instead of the $d$th symmetric power and project down to the symmetric power afterwards.
The tensor power $\tensor^d(\tensor^3 \IC^n)$ is a $\GL_n^3$-representation and for every partition triple $(\la,\mu,\nu)$ we know a generating set (even a basis)
of the highest weight vector space $\HWV_{\la,\mu,\nu}(\tensor^d(\tensor^3\IC^n))$.

We now construct an element in $\HWV_{\la,\mu,\nu}(\tensor^d(\tensor^3\IC^n))$.
Let $\la$, $\mu$, and $\nu$ have $d$ boxes each.
Let $e_1,e_2,\ldots,e_n$ denote the standard basis of $\IC^n$.
For $i \in \IN$ let
\[
\widehat i := e_1 \wedge e_2 \wedge \cdots \wedge e_i \in \tensor^i \IC^n.
\]
For $\nu$ with column lengths $\transpose\nu_1,\transpose\nu_2,\ldots$ we use this notation to define skew symmetric tensors
$\widehat{\transpose\nu_1}, \widehat{\transpose\nu_2}, \ldots$ and their tensor product
\[
\widehat \nu := \widehat{\transpose\nu_1} \otimes \widehat{\transpose\nu_2} \otimes \cdots \in \tensor^d \IC^n.
\]
It is readily checked that $\widehat\la \otimes \widehat \mu \otimes \widehat \nu \in \HWV_{\la,\mu,\nu}(\tensor^3(\tensor^d \IC^n)).$
We define $\widehat{\la,\mu,\nu} \in \HWV_{\la,\mu,\nu}(\tensor^d(\tensor^3 \IC^n))$ to be the image of $\widehat\la \otimes \widehat \mu \otimes \widehat \nu$
under the isomorphism $\HWV_{\la,\mu,\nu}(\tensor^d(\tensor^3 \IC^n)) \simeq \HWV_{\la,\mu,\nu}(\tensor^3(\tensor^d \IC^n))$,
which is just reordering the tensor factors.

We will now study a graphical interpretation of the contraction of $\widehat{\la,\mu,\nu}$ with other tensors.
For a list of vectors $v_1,\ldots,v_m \in \IC^n$, $m \leq n$, let the \emph{evaluation} $\sdet(v_1,\ldots,v_m)$ denote the determinant of the $m \times m$ matrix obtained from the column vectors $v_1,\ldots,v_m$ by taking only 
the first $m$ entries of each $v_i$.
For example
\[
\sdet\begin{pmatrix}
1 & 2 \\
0 & -1 \\
3 & 1
\end{pmatrix} = \det\begin{pmatrix}
1 & 2 \\
0 & -1
\end{pmatrix}=-1.
\]
We now construct a hypergraph $H$ with vertex set $\{1,2,\ldots,d\}$ as follows.
There are three layers of hyperedges corresponding to $\la$, $\mu$, and $\nu$, respectively.
Every vertex lies in exactly three hyperedges, one from each layer.
Every hyperedge in layer 1 corresponds to a column in $\la$, analogously for layer 2 and $\mu$ and for layer~3 and~$\nu$.
Filling a Young tableau of shape $\la$ columnwise from top to bottom gives the hyperedges of the first layer,
for example for $\la=(4,3,2,1)$ we would fill columnwise and obtain
\[
\young(158\Zehn,269,37,4),
\]
so the hyperedges of layer 1 are arranged as follows:
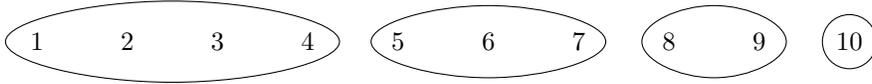
\[
\begin{tikzpicture}[scale=1.2]
\node at ( 1,0) {1};
\node at ( 2,0) {2};
\node at ( 3,0) {3};
\node at ( 4,0) {4};
\node at ( 5,0) {5};
\node at ( 6,0) {6};
\node at ( 7,0) {7};
\node at ( 8,0) {8};
\node at ( 9,0) {9};
\node at (10,0) {10};
\draw (0.65,0.0) .. controls (0.649999,0.224999) and (1.574999,0.45) .. (2.5,0.45) .. controls (3.425000,0.45) and (4.35,0.225000) .. (4.35,0.0) .. controls (4.35,-0.22500) and (3.425,-0.45000) .. (2.5,-0.45) .. controls (1.575000,-0.44999) and (0.650000,-0.22499) .. (0.65,0.0) -- cycle;
\draw (4.7,0.0) .. controls (4.7,0.199999) and (5.350000,0.400000) .. (6.0,0.4) .. controls (6.649999,0.399999) and (7.300000,0.199999) .. (7.3,0.0) .. controls (7.299999,-0.19999) and (6.65,-0.39999) .. (6.0,-0.4) .. controls (5.35,-0.40000) and (4.7,-0.19999) .. (4.7,0.0) -- cycle;
\draw (7.7,0.0) .. controls (7.699999,0.199999) and (8.1,0.400000) .. (8.5,0.4) .. controls (8.9,0.399999) and (9.3,0.199999) .. (9.3,0.0) .. controls (9.3,-0.19999) and (8.9,-0.39999) .. (8.5,-0.4) .. controls (8.1,-0.40000) and (7.700000,-0.19999) .. (7.7,0.0) -- cycle;
\draw (10,0) circle (0.3cm);
\end{tikzpicture}
\]
We do the same for $\mu$ and $\nu$. For example, if $\la=\mu=\varrho(4)$ and $\nu=(5,3,1,1)$ we obtain the following hypergraph:
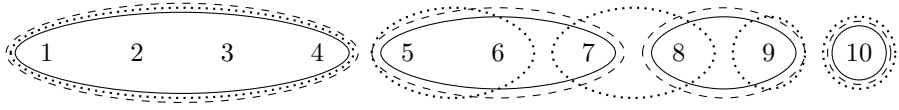
\[
\begin{tikzpicture}[scale=1.2]
\node at ( 1,0) {1};
\node at ( 2,0) {2};
\node at ( 3,0) {3};
\node at ( 4,0) {4};
\node at ( 5,0) {5};
\node at ( 6,0) {6};
\node at ( 7,0) {7};
\node at ( 8,0) {8};
\node at ( 9,0) {9};
\node at (10,0) {10};
\draw[thick,dotted] (0.6,0.0) .. controls (0.6,0.249999) and (1.549999,0.5) .. (2.5,0.5) .. controls (3.45,0.5) and (4.400000,0.249999) .. (4.4,0.0) .. controls (4.399999,-0.24999) and (3.45,-0.5) .. (2.5,-0.5) .. controls (1.549999,-0.49999) and (0.6,-0.24999) .. (0.6,0.0) -- cycle;
\draw[thick,dotted] (4.6,0.0) .. controls (4.6,0.249999) and (5.050000,0.5) .. (5.5,0.5) .. controls (5.949999,0.5) and (6.4,0.249999) .. (6.4,0.0) .. controls (6.4,-0.24999) and (5.95,-0.5) .. (5.5,-0.5) .. controls (5.05,-0.49999) and (4.6,-0.24999) .. (4.6,0.0) -- cycle;
\draw[thick,dotted] (6.6,0.0) .. controls (6.599999,0.249999) and (7.049999,0.5) .. (7.5,0.5) .. controls (7.950000,0.5) and (8.4,0.249999) .. (8.4,0.0) .. controls (8.4,-0.24999) and (7.949999,-0.5) .. (7.5,-0.5) .. controls (7.050000,-0.49999) and (6.600000,-0.24999) .. (6.6,0.0) -- cycle;
\draw[thick,dotted] (9,0) circle (0.4cm);
\draw[thick,dotted] (10,0) circle (0.4cm);
\draw[dashed] (0.55,0.0) .. controls (0.549999,0.275) and (1.524999,0.550000) .. (2.5,0.55) .. controls (3.475000,0.549999) and (4.45,0.274999) .. (4.45,0.0) .. controls (4.45,-0.27499) and (3.475,-0.55) .. (2.5,-0.55) .. controls (1.525,-0.55) and (0.550000,-0.275) .. (0.55,0.0) -- cycle;
\draw (0.65,0.0) .. controls (0.649999,0.224999) and (1.574999,0.45) .. (2.5,0.45) .. controls (3.425000,0.45) and (4.35,0.225000) .. (4.35,0.0) .. controls (4.35,-0.22500) and (3.425,-0.45000) .. (2.5,-0.45) .. controls (1.575000,-0.44999) and (0.650000,-0.22499) .. (0.65,0.0) -- cycle;
\draw[dashed] (4.6,0.0) .. controls (4.599999,0.249999) and (5.299999,0.5) .. (6.0,0.5) .. controls (6.700000,0.5) and (7.4,0.249999) .. (7.4,0.0) .. controls (7.4,-0.24999) and (6.699999,-0.5) .. (6.0,-0.5) .. controls (5.300000,-0.49999) and (4.600000,-0.24999) .. (4.6,0.0) -- cycle;
\draw (4.7,0.0) .. controls (4.7,0.199999) and (5.350000,0.400000) .. (6.0,0.4) .. controls (6.649999,0.399999) and (7.300000,0.199999) .. (7.3,0.0) .. controls (7.299999,-0.19999) and (6.65,-0.39999) .. (6.0,-0.4) .. controls (5.35,-0.40000) and (4.7,-0.19999) .. (4.7,0.0) -- cycle;
\draw[dashed] (7.6,0.0) .. controls (7.599999,0.249999) and (8.049999,0.5) .. (8.5,0.5) .. controls (8.950000,0.5) and (9.400000,0.249999) .. (9.4,0.0) .. controls (9.399999,-0.24999) and (8.949999,-0.5) .. (8.5,-0.5) .. controls (8.050000,-0.49999) and (7.600000,-0.24999) .. (7.6,0.0) -- cycle;
\draw (7.7,0.0) .. controls (7.699999,0.199999) and (8.1,0.400000) .. (8.5,0.4) .. controls (8.9,0.399999) and (9.3,0.199999) .. (9.3,0.0) .. controls (9.3,-0.19999) and (8.9,-0.39999) .. (8.5,-0.4) .. controls (8.1,-0.40000) and (7.700000,-0.19999) .. (7.7,0.0) -- cycle;
\draw[dashed] (10,0) circle (0.35cm);
\draw (10,0) circle (0.3cm);
\end{tikzpicture}
\]
where layer 2 is drawn with dashed lines and layer 3 with dotted lines.
Let us call this hypergraph~$H$.
Let $E_i(H)$ denote the set of hyperedges in layer~$i$, $1 \leq i \leq 3$.
For a hyperedge $S$ let $S_1$ denote its smallest entry, $S_2$ denote its second smallest entry, and so on.
Let $\ell(S)$ denote the number of vertices in the hyperedge~$S$.
Let $\circ$ denote the contraction of tensors.
The main property of $\widehat{\la,\mu,\nu}$ is the following, which can be readily checked by calculation.
For all $a_i, b_i, c_i \in \IC^n$, $1 \leq i \leq n$, we have
\begin{eqnarray}\label{eq:sldetpi}
&& \widehat{\la,\mu,\nu} \circ \Big( (a_1 \otimes b_1 \otimes c_1) \otimes (a_2 \otimes b_2 \otimes c_2) \otimes \cdots \otimes (a_d \otimes b_d \otimes c_d) \Big) \nonumber\\
&=& \prod_{S \in E_1(H)} \sdet(a_{S_1}, a_{S_2}, \ldots, a_{S_{\ell(S)}} ) \nonumber \\
&\cdot& \prod_{S \in E_2(H)} \sdet(b_{S_1}, b_{S_2}, \ldots, b_{S_{\ell(S)}} )  \\
&\cdot& \prod_{S \in E_3(H)} \sdet(c_{S_1}, c_{S_2}, \ldots, c_{S_{\ell(S)}} ). \nonumber
\end{eqnarray}
Let us summarize the key properties of $H$ in the following definition.
\begin{definition}
Let $d:=|\la|=|\mu|=|\nu|$.
A \emph{Young hypergraph} $H$ of type $(\la,\mu,\nu)$ is a hypergraph with $d$ vertices such that
\begin{itemize}
\item There are three layers of hyperedges corresponding to $\nu$, $\mu$, and $\nu$, respectively.
\item Every vertex lies in exactly three hyperedges, one from each layer.
\item There is a bijection between the vertices of $H$ and the boxes in $\la$ such that
two vertices lie in a common hyperedge in layer 1 iff the corresponding boxes in $\la$ lie in the same column.
Analogously for layer 2 and $\mu$ and for layer 3 and~$\nu$. \hfill$\blacksquare$
\end{itemize}
\end{definition}
The crucial point is the following.
The fact that the actions of the groups $\GL_n^3$ and $\aS_d \times \aS_d \times \aS_d$ on $\tensor^d(\tensor^3\IC^n)$ commute implies that
we can use \emph{any Young hypergraph $H$} instead of the one we just constructed and in this manner we can define
a highest weight vector $\widehat{\la,\mu,\nu}_H$ whose tensor contraction works exactly as in equation~\eqref{eq:sldetpi}.

The group $\aS_d$ acts on $\tensor^d (\tensor^3 \IC^n)$ by rearranging the tensor factors.
To prove that $\HWV_{\la,\mu,\nu}(\Sym^d(\tensor^3 \IC^n)) \neq 0$ it is sufficient to create a Young hypergraph $H$ of type $(\la,\mu,\nu)$ such that
the projection of $\widehat{\la,\mu,\nu}_H$ to $\HWV_{\la,\mu,\nu}(\Sym^d(\tensor^3 \IC^n))$ is nonzero.
For this, it is sufficient to find a symmetric tensor $v \in \Sym^d(\tensor^3\IC^n)$ such that
$\widehat{\la,\mu,\nu}_H \circ v \neq 0$.
To prove Theorem~\ref{thm:main}
it remains to construct a Young hypergraph of type $(\varrho(n),\varrho(n),\nu)$
and a tensor $v\in \Sym^d(\tensor^3\IC^n)$ with the property $\widehat{\la,\mu,\nu}_H \circ v \neq 0$.
We will do so in the next sections.

\section{Construction of the Young Hypergraph}
For the rest of this paper we fix $n$, we let $\varrho:=\varrho(n)$, $d:=\frac{n(n+1)}2$, and we fix $\nu$ a partition of $d$ such that
$\nu$ dominates $\varrho$.
We construct a Young hypergraph $H$ of type $(\varrho,\varrho,\nu)$ with $d$ vertices as follows.
We start by defining a finite set $\Delta_n := \{(r,c) \in \IN \times \IN \mid 1 \leq r,c \leq n, \ r+c \leq n+1\}$ of $d$ points in the planar grid.
For example for $n=4$ the arrangement $\Delta_n$ can be depicted as follows:
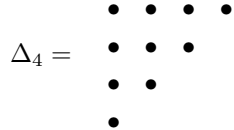
\[
\raisebox{1cm}{$\Delta_4=$} \quad
\begin{tikzpicture}[scale=0.5]
\node at (1,-1) {$\bullet$};
\node at (2,-1) {$\bullet$};
\node at (3,-1) {$\bullet$};
\node at (4,-1) {$\bullet$};
\node at (1,-2) {$\bullet$};
\node at (2,-2) {$\bullet$};
\node at (3,-2) {$\bullet$};
\node at (1,-3) {$\bullet$};
\node at (2,-3) {$\bullet$};
\node at (1,-4) {$\bullet$};
\end{tikzpicture}
\]
The set $\Delta_n$ forms the vertex set of the Young hypergraph~$H$.
We will see that the numbering of the vertex set can be done in any way, so we omit it.
The hyperedges for the first layer of $H$ are formed by the rows and the hyperedges of the second layer are given by the columns, so for example for $n=4$ we have the following picture.
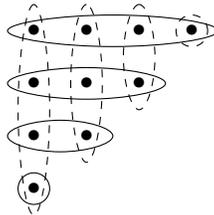
\[
\begin{tikzpicture}[scale=0.7]
\node at (1,-1) {$\bullet$};
\node at (2,-1) {$\bullet$};
\node at (3,-1) {$\bullet$};
\node at (4,-1) {$\bullet$};
\node at (1,-2) {$\bullet$};
\node at (2,-2) {$\bullet$};
\node at (3,-2) {$\bullet$};
\node at (1,-3) {$\bullet$};
\node at (2,-3) {$\bullet$};
\node at (1,-4) {$\bullet$};
\draw (0.5,-1.0) .. controls (0.499999,-0.85000) and (1.5,-0.7) .. (2.5,-0.7) .. controls (3.5,-0.7) and (4.5,-0.85000) .. (4.5,-1.0) .. controls (4.5,-1.15) and (3.5,-1.3) .. (2.5,-1.3) .. controls (1.5,-1.3) and (0.500000,-1.15) .. (0.5,-1.0) -- cycle;
\draw (0.5,-2.0) .. controls (0.5,-1.85) and (1.25,-1.69999) .. (2.0,-1.7) .. controls (2.75,-1.70000) and (3.5,-1.85000) .. (3.5,-2.0) .. controls (3.5,-2.14999) and (2.75,-2.3) .. (2.0,-2.3) .. controls (1.25,-2.3) and (0.5,-2.15) .. (0.5,-2.0) -- cycle;
\draw (0.5,-3.0) .. controls (0.5,-2.84999) and (1.0,-2.7) .. (1.5,-2.7) .. controls (2.0,-2.7) and (2.5,-2.84999) .. (2.5,-3.0) .. controls (2.5,-3.15000) and (2.0,-3.3) .. (1.5,-3.3) .. controls (0.999999,-3.3) and (0.5,-3.15000) .. (0.5,-3.0) -- cycle;
\draw (1,-4) circle (0.3cm);
\draw[dashed] (1.0,-0.5) .. controls (0.850000,-0.49999) and (0.7,-1.5) .. (0.7,-2.5) .. controls (0.7,-3.5) and (0.850000,-4.5) .. (1.0,-4.5) .. controls (1.15,-4.5) and (1.3,-3.5) .. (1.3,-2.5) .. controls (1.3,-1.5) and (1.15,-0.50000) .. (1.0,-0.5) -- cycle;
\draw[dashed] (2.0,-0.5) .. controls (1.85,-0.5) and (1.699999,-1.25) .. (1.7,-2.0) .. controls (1.700000,-2.75) and (1.850000,-3.5) .. (2.0,-3.5) .. controls (2.149999,-3.5) and (2.3,-2.75) .. (2.3,-2.0) .. controls (2.3,-1.25) and (2.15,-0.5) .. (2.0,-0.5) -- cycle;
\draw[dashed] (3.0,-0.5) .. controls (2.849999,-0.5) and (2.7,-1.0) .. (2.7,-1.5) .. controls (2.7,-2.0) and (2.849999,-2.5) .. (3.0,-2.5) .. controls (3.150000,-2.5) and (3.3,-2.0) .. (3.3,-1.5) .. controls (3.3,-0.99999) and (3.150000,-0.5) .. (3.0,-0.5) -- cycle;
\draw[dashed] (4,-1) circle (0.3cm);
\end{tikzpicture}
\]
Besides a row number $r(x) \in \{1,2,\ldots,n\}$ and a column number $c(x) \in \{1,2,\ldots,n\}$,
each vertex $x$ has a value $\beta(x) \in \{1,2,\ldots,n\}$, which we define as $\beta(x)=n+2-r(x)-c(x)$.
The $\beta$ value can be interpreted as the distance from the diagonal edge of the triangular array $\Delta_n$.
For example in the case $n=4$ the $\beta$ values are as follows:
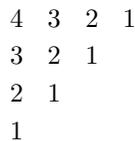
\[
\begin{tikzpicture}[scale=0.5]
\node at (1,-1) {4};
\node at (2,-1) {3};
\node at (3,-1) {2};
\node at (4,-1) {1};
\node at (1,-2) {3};
\node at (2,-2) {2};
\node at (3,-2) {1};
\node at (1,-3) {2};
\node at (2,-3) {1};
\node at (1,-4) {1};
\end{tikzpicture}
\]
Recall that the sizes of the hyperedges of the third layer correspond to the column lengths of~$\nu$.
We now choose the hyperedges of the third layer in a way that for each hyperedge the $\beta$ values of all its vertices are distinct.
The key insight is that this is possible!
Indeed, the following lemma says that this can be done iff $\nu$ dominates~$\varrho$.
\begin{lemma}
The following two statements are equivalent:
\begin{itemize}
\item In $\Delta_n$ there exists a partition of the vertex set into hyperedges of sizes given by the column lengths of $\nu$ such that
for each hyperedge the $\beta$ values of all vertices in the hyperedge are distinct.
\item $\nu$ dominates $\varrho$.
\end{itemize}
\end{lemma}
\begin{proof}
A \emph{filling} of shape $\nu$ and content $\gamma \in \IN^n$ is an assignment of numbers to the boxes of $\nu$
such that each entry $i$ appears exactly $\gamma_i$ times.
A filling is \emph{semistandard} if the entries are increasing along each column and nondecreasing along each row.
We prove that the following four statements are all equivalent:
\begin{enumerate}[(1)]
\item In $\Delta_n$ there exists a partition of the vertex set into hyperedges of sizes given by the column lengths of $\nu$ such that
for each hyperedge the $\beta$ values of all vertices in the hyperedge are distinct.
\item There exists a filling of $\nu$ with content $\varrho$ such that there is no column with two coinciding entries.
\item There exists a semistandard filling of $\nu$ with content $\varrho$.
\item $\nu$ dominates $\varrho$.
\end{enumerate}
The statement that (3) is equivalent to (4) is known as the Gale-Ryser theorem, see e.g.\ \cite[p.~457, Ex.~A.11]{FH:91}, \cite[p.~26, Ex.~2]{fult:97}, or~\cite[I.7 Exa.~9]{MaD:95}.
Clearly (3) implies (2). But (2) also implies (3) by straightening the filling, see e.g.\ \cite[p.~110]{fult:97}.
It remains to show that (1) iff (2).
From a partition of the vertex set into hyperedges we obtain a filling of shape $\nu$ by constructing for each hyperedge a column
whose entries are exactly the $\beta$ values of the vertices in the hyperedge.
On the other hand, from a filling we get a partition of the vertex set into hyperedges by constructing
for each column a hyperedge that has vertices whose $\beta$ values are exactly the values in the column.
For example, for $\nu=(5,3,1,1)$ we can find a filling
\[
\young(11112,223,3,4)
\]
from which we can construct the third layer such that the $\beta$ values of the vertices in a hyperedge are exactly the numbers appearing in a column:
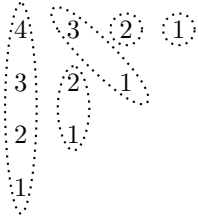
\[
\begin{tikzpicture}[scale=0.7]
\node at (1,-1) {4};
\node at (2,-1) {3};
\node at (3,-1) {2};
\node at (4,-1) {1};
\node at (1,-2) {3};
\node at (2,-2) {2};
\node at (3,-2) {1};
\node at (1,-3) {2};
\node at (2,-3) {1};
\node at (1,-4) {1};
\draw[thick,dotted] (1.0,-0.5) .. controls (0.850000,-0.49999) and (0.7,-1.5) .. (0.7,-2.5) .. controls (0.7,-3.5) and (0.850000,-4.5) .. (1.0,-4.5) .. controls (1.15,-4.5) and (1.3,-3.5) .. (1.3,-2.5) .. controls (1.3,-1.5) and (1.15,-0.50000) .. (1.0,-0.5) -- cycle;
\draw[thick,dotted] (2.0,-1.7) .. controls (1.85,-1.70000) and (1.699999,-2.09999) .. (1.7,-2.5) .. controls (1.700000,-2.90000) and (1.850000,-3.30000) .. (2.0,-3.3) .. controls (2.149999,-3.29999) and (2.3,-2.9) .. (2.3,-2.5) .. controls (2.3,-2.1) and (2.15,-1.69999) .. (2.0,-1.7) -- cycle;
\draw[thick,dotted] (1.7,-1.0) .. controls (1.45,-0.6) and (1.599999,-0.44999) .. (2.0,-0.7) .. controls (2.400000,-0.95) and (3.05,-1.6) .. (3.3,-2.0) .. controls (3.55,-2.4) and (3.4,-2.55) .. (3.0,-2.3) .. controls (2.6,-2.05) and (1.95,-1.4) .. (1.7,-1.0) -- cycle;
\draw[thick,dotted] (3,-1) circle (0.3cm);
\draw[thick,dotted] (4,-1) circle (0.3cm);
\end{tikzpicture}
\]
\end{proof}

\section{Contraction with the Symmetric Tensor}
Let $e_1,e_2,\ldots,e_n$ denote the standard basis of $\IC^n$.
Choose generic vectors $c_1, \ldots, c_n \in \IC^n$.
Let
\[
\bar v := (e_1 \otimes e_1 \otimes c_1)^{\otimes n} \otimes (e_2 \otimes e_2 \otimes c_2)^{\otimes n-1} \otimes \cdots \otimes (e_n \otimes e_n \otimes c_n)^{\otimes 1} \in \tensor^d(\tensor^3 \IC^n).
\]
Define $v := \sum_{\sigma \in \aS_d} \sigma(\bar v) \in \Sym^d(\tensor^3\IC^n)$.
As described at the end of section~\ref{sec:newinterpretationofkron} it remains to show that $\widehat{\la,\mu,\nu}_H \circ v \neq 0$.
By linearity we have
\begin{equation}\label{eq:perm}
\widehat{\la,\mu,\nu}_H \circ v = \sum_{\sigma \in \aS_d} \widehat{\la,\mu,\nu}_H \circ \sigma(\bar v).
\end{equation}
Note that that stabilizer of $\bar v$ in $\aS_d$ is the Young subgroup $\aS_n \times \aS_{n-1} \times \cdots \times \aS_1 \subseteq \aS_d$,
so actually \eqref{eq:perm} is a sum of $d!/(n! (n-1)! \cdots 2!)$ summands, each with coefficient $\theta := n! (n-1)! \cdots 2!$.
Let $M$ denote the set of all mappings $\tau:\{1,\ldots,d\}\to\{1,\ldots,n\}$ such that the cardinality $|\tau^{-1}(i)|$ of the preimage of $i$ is $n+1-i$ for all $1 \leq i \leq n$.
Then we can rewrite \eqref{eq:perm} as
\begin{equation}\label{eq:permmaps}
\widehat{\la,\mu,\nu}_H \circ v = \theta \sum_{\tau \in M} \widehat{\la,\mu,\nu}_H \circ (e_{\tau(1)} \otimes e_{\tau(1)} \otimes c_{\tau(1)}) \otimes \cdots \otimes (e_{\tau(d)} \otimes e_{\tau(d)} \otimes c_{\tau(d)}).
\end{equation}
The map $\tau$ can be thought of as placing numbers $1$ up to $n$ on the vertices of $\Delta_n$, each number $i$ exactly $n+1-i$ times.
The key observation we want to prove is that there is exactly one nonzero summand in \eqref{eq:permmaps}, namely the one where $\tau=\beta$.
We now give strong restrictions on how $\tau$ can look like in the case where the summand corresponding to $\tau$ in \eqref{eq:permmaps} is nonzero.
The main argument we use is that for every hyperedge $\{x_1,x_2,\ldots,x_k\}$ the evaluation
$\sdet(e_{\tau(x_1)}, e_{\tau(x_2)}, \ldots, e_{\tau(x_k)})$ is nonzero iff the list $(\tau(x_1),\ldots,\tau(x_k))$ is a permutation
of $(1,2,\ldots,k)$. We refer to this fact as $(\ast)$.
The fact that there is only a single vertex $x$ in the bottom row implies that $\tau(x)=1$ by $(\ast)$,
because otherwise this singleton hyperedge in the row contributes a zero factor in the evaluation of the first layer.
For the vertex $y$ directly above $x$ by applying $(\ast)$ we see that
we cannot set $\tau(y)=1$, because the first column is a hyperedge in the second layer.
But from $(\ast)$ we know that in the row of $y$ the map $\tau$ has to place exactly the numbers $1$ and $2$,
so we must set $\tau(y)=2$ and $\tau(y')=1$ for the right neighbor $y'$ of~$y$.
This argument continues up through all rows until we see that at any vertex $x$ we can only place $\tau(x)=\beta(x)$.
The determinants of all hyperedges of the first and second layer are determinants of permutation matrices, so they are either $1$ or $-1$, but certainly nonzero.
Since the hyperedges in the third layer have the property that no hyperedge has two vertices with the same $\beta$ value,
and since the $c_i$ were chosen generically, the determinants of the hyperedges of the third layer are all nonzero.
This finishes the proof of Theorem~\ref{thm:main}.

\section{Hooks}\label{sec:hooks}
In this section we use Theorem~\ref{thm:main} to prove the Saxl conjecture for hooks, see Corollary~\ref{cor:saxlhooks}.
Interestingly, the proof of Corollary~\ref{cor:saxlhooks} in \cite{ppv:13} uses a very different technique.

Let $n \times 1$ denote the partition with $n$ boxes in a single column
and let $1 \times n$ denote the partition with $n$ boxes in a single row.
The addition of partitions is defined as the addition of their parts.
A partition $\nu$ is called a \emph{hook} if $\nu$ can be written as $\nu=1 \times n + m\times 1$ for some $n,m \in \IN_{\geq 0}$.
\begin{corollary}\label{cor:saxlhooks}
Let $d := n(n+1)/2$.
For every hook $\nu$ with $d$ boxes we have $\kronk{\varrho(n)}{\varrho(n)}{\nu}>0$.
\end{corollary}
\begin{proof}
We use induction on $d$, where the base case $d=1$ is trivial.
If $\nu$ has at most $n$ columns, then the statement holds by Theorem~\ref{thm:main} because $\nu$ is dominated by $\varrho(n)$.
If $\nu$ has more than $n$ columns, then
we can obtain a partition $\bar\nu$ by removing $n$ boxes from the first row of $\nu$, so $1 \times n + \bar\nu = \nu$.
By induction hypothesis
$\kronk{\varrho(n-1)}{\varrho(n-1)}{\bar\nu}>0$.
Since $\kronk{n \times 1}{n \times 1}{1 \times n}=1 > 0$ and $n \times 1 + \varrho(n-1) = \varrho(n)$,
the semigroup property (see \cite[Thm.~3.1]{chm:05} or \cite[Prop.~4.4.10]{ike:12b} for a different viewpoint) implies
$\kronk{\varrho(n)}{\varrho(n)}{\nu} > 0$.
\end{proof}

\bibliographystyle{alpha}

\end{document}